\documentclass[10pt,reqno]{amsart}
\usepackage[british]{babel}

\usepackage{setspace}
\usepackage{amssymb}
\usepackage{amsmath}
\usepackage{amsthm}
\usepackage{array}
\usepackage{enumitem}
\usepackage[colorlinks,citecolor=blue,urlcolor=black,linkcolor=red,linktocpage]{hyperref}
\usepackage{mathrsfs}

\newtheorem{thm}{Theorem}[section]

\newtheorem{cor}[thm]{Corollary}
\newtheorem{lem}[thm]{Lemma}

\theoremstyle{definition}

\newtheorem{exmpl}[thm]{Example}

\newtheorem{definition}[thm]{Definition}

\newtheorem{remark}[thm]{Remark}

\newtheorem{question}[thm]{Question}

\renewcommand{\epsilon}{\varepsilon}
\renewcommand{\phi}{\varphi}
\newcommand{\defeq}{\mathrel{\mathop:}=}

\DeclareMathOperator{\Aut}{Aut}

\makeatletter
\def\moverlay{\mathpalette\mov@rlay}
\def\mov@rlay#1#2{\leavevmode\vtop{%
		\baselineskip\z@skip \lineskiplimit-\maxdimen
		\ialign{\hfil$\m@th#1##$\hfil\cr#2\crcr}}}
\newcommand{\charfusion}[3][\mathord]{
	#1{\ifx#1\mathop\vphantom{#2}\fi
		\mathpalette\mov@rlay{#2\cr#3}
	}
	\ifx#1\mathop\expandafter\displaylimits\fi}
\makeatother

\DeclareMathOperator{\diam}{diam}

\begin{document}
%%%%%%%%%%%%%%%%%%%%%%%%%%%%%%%%

%\onehalfspace

\setlist{noitemsep}

\author{Vladimir G. Pestov}
\address{V.P., Department of Mathematics and Statistics, University of Ottawa, 585 King Edward Ave., Ottawa, Ontario, Canada K1N 6N5 }
\curraddr{Departamento de Matem\'atica, Universidade Federal de Santa Catarina, Trindade, Florian\'opolis, SC, 88.040-900, Brazil}
\email{vpest283@uottawa.ca}

\author{Friedrich Martin Schneider}
\address{F.M.S., Institut f\"ur Algebra, TU Dresden, 01062 Dresden, Germany }
%\curraddr{Department of Mathematics, University of Auckland, Private Bag 92019, NZ}
\email{martin.schneider@tu-dresden.de}

\title{On amenability and groups of measurable maps}
\date{August 29, 2017}

\begin{abstract} 
	We show that if $G$ is an amenable topological group, then the topological group $L^{0}(G)$ of strongly measurable maps from $([0,1],\lambda)$ into $G$ endowed with the topology of convergence in measure is whirly amenable, hence extremely amenable. Conversely, we prove that a topological group $G$ is amenable if $L^{0}(G)$ is.
\end{abstract}

\maketitle

%%%%%%%%%%%%%%%%%%%%%%%%%%%%%%%%%%%%%%%%%%
%%%%%%%%%%%%%%%%%%%%%%%%%%%%%%%%%%%%%%%%%%

\section{Introduction}

In this note we study measurable maps taking values in topological groups. Let us recall that a map $f \colon X \to Y$ from a compact Hausdorff space $X$ carrying a regular Borel probability measure $\mu$ into a topological space $Y$ is \emph{strongly $\mu$-measurable}\footnote{Alternative terms used in the literature are \emph{$\mu$-measurable in the sense of Bourbaki}~\cite[p.~357]{gaal}, \emph{Lusin $\mu$-measurable}~\cite{schwartz}, or \emph{$\mu$-almost continuous}~\cite{fremlin}.} if, for every $\epsilon > 0$, there exists a closed subset $A \subseteq X$ with $\mu (X\setminus A) \leq \epsilon$ such that $f\vert_{A} \colon A \to Y$ is continuous. Evidently, if $f$ is strongly $\mu$-measurable, then it is $\mu$-measurable, in the sense that $f^{-1}(B)$ is $\mu$-measurable for every Borel $B \subseteq G$. In case $Y$ is metrizable~\cite[Theorem~2B]{fremlin}, the map $f$ is strongly $\mu$-measurable if and only if $f$ is $\mu$-measurable.

Let $G$ be a topological group. Considering the Lebesgue measure $\lambda$ on the real interval $[0,1]$, we define $L^{0}(G)$ to be the set of all $\lambda$-equivalence classes of strongly $\lambda$-measurable maps from $[0,1]$ into $G$. Equipping $L^{0}(G)$ with the group structure inherited from $G$ and the topology of convergence in measure, we obtain a topological group. We refer to Section~\ref{section:final} for more details on~$L^{0}(G)$.

The aim of this note is to prove the following, thus answering a question raised in~\cite{giordanopestov}.

\begin{thm}\label{theorem:main} Let $G$ be a topological group. The following are equivalent. \begin{enumerate}
	\item[$(1)$] $G$ is amenable.
	\item[$(2)$] $L^{0}(G)$ is amenable.
	\item[$(3)$] $L^{0}(G)$ is extremely amenable.
	\item[$(4)$] $L^{0}(G)$ is whirly amenable.
\end{enumerate} \end{thm}

In particular, Theorem~\ref{theorem:main} provides a source of extremely amenable groups. The historically initial example, of a group with no nontrivial unitary representations, hence extremely amenable (since it is also amenable), was obtained by Herer and Christensen~\cite{HererChristensen}, as a group of measurable functions with values in the circle, however with a pathological submeasure in place of the Lebesgue measure.

For $G$ the circle group, the extreme amenability of $L^{0}(G)$ (relative to the Lebesgue measure) was proved by Glasner, although only published many years later~\cite{glasner98}, and, as mentioned in the article, at the same time and independently it was proved by Furstenberg and Benjy Weiss, unpublished. (Of course, this group has plenty of unitary representations, recently classified by Solecki~\cite{solecki}.) The result was later generalized to amenable locally compact groups $G$ \cite{pestov02,pestov10}.

The proof of Theorem~\ref{theorem:main} (see Section~\ref{section:final}) combines a recent result from~\cite{SchneiderThom} with a measure concentration argument. The application of measure concentration phenomena for proving extreme amenability was initiated by Gromov and Milman~\cite{GromovMilman}. We note that for non-atomic submeasures the concentration technique cannot be used, because there is in general no concentration, so special combinatorial techniques must be applied instead, see Farah-Solecki~\cite{FarahSolecki} and Sabok~\cite{sabok}, and the open questions therein. It would be interesting to see if the amenability criterion from~\cite{SchneiderThom} can be used to extend these combinatorial techniques to all amenable topological groups.

\section{Preliminaries}

In this section we fix some notation and briefly review some basic facts concerning uniform spaces, topological groups, and amenability.

\subsection{Notation} Let us make some initial remarks about notation. Let $X$ be a set. As usual, we will denote by $\ell^{\infty}(X)$ the Banach space of all bounded real-valued functions on $X$ equipped with the supremum norm \begin{displaymath}
	\Vert f \Vert_{\infty} \defeq \sup \{ \vert f(x) \vert \mid x \in X \} \qquad (f \in \ell^{\infty}(X)) .
\end{displaymath} We denote by $\mathrm{Pr}(X)$ the set of all probability measures on $X$ and by $\mathrm{Pr}_{\mathrm{fin}}(X)$ the set of all finitely supported probability measures on $X$. If $X$ is a topological space, then we will denote by $\mathrm{Pr}_{\mathrm{B}}(X)$ the set of all Borel probability measures on $X$.

\subsection{UEB topology} Let $X$ be a uniform space. We consider the Banach space $\mathrm{UC}_{b}(X)$ of all bounded uniformly continuous real-valued functions on $X$ endowed with the supremum norm. A subset $H \subseteq \mathrm{UC}_{b}(X)$ is called \emph{UEB} (short for \emph{uniformly equicontinuous bounded}) if $H$ is uniformly equicontinuous and $\Vert \cdot \Vert_{\infty}$-bounded. The set $\mathrm{UEB}(X)$ of all UEB subsets of $\mathrm{UC}_{b}(X)$ constitutes a convex vector bornology on the vector space $\mathrm{UC}_{b}(X)$. The \emph{UEB topology} on the continuous dual~$\mathrm{UC}_{b}(X)'$ is defined as the topology of uniform convergence on UEB subsets of $\mathrm{UC}_{b}(X)$. This is a locally convex linear topology on the vector space $\mathrm{UC}_{b}(X)'$ containing the weak*-topology, i.e., the initial topology generated by the maps $\mathrm{UC}_{b}(X)' \to \mathbb{R}, \, \mu  \mapsto \mu (f)$ where $f \in \mathrm{UC}_{b}(X)$. A description of the UEB topology in terms of semi-norms emerging from compatible pseudo-metrics on $X$ can be found in~\cite[Section~2]{SchneiderThom}. We note that the real vector space $\mathbb{R}X$ over the base set $X$ equipped with the UEB topology inherited from~$\mathrm{UC}_{b}(X)'$ via the natural embedding is related to the free locally convex real topological vector space over the uniform space $X$~\cite{raikov}, although the two spaces are different already for $X={\mathbb R}$; rather, the space $\mathbb{R}X$ has a similar universal property with regard to continuous linear maps with bounded image. For more details on the UEB topology we refer to~\cite{PachlBook,NeufangPachlPekka,PachlSteprans}.

\subsection{Topological groups} Let $G$ be a topological group. Let $\mathscr{U}(G)$ denote the neighborhood filter of the neutral element in $G$ and let $\mathscr{U}_{o}(G)$ be the set of all open sets in $\mathscr{U}(G)$. Let us endow $G$ with its \emph{right uniformity} defined by the basic entourages of the form \begin{displaymath}
	\left\{ (x,y) \in G \times G \left\vert \, xy^{-1} \in U \right\} \qquad (U \in \mathscr{U}(G)) . \right.
\end{displaymath} Referring to this uniform structure, we denote by $\mathrm{RUC}_{b}(G)$ the space of all bounded uniformly continuous real-valued function on $G$ and by $\mathrm{RUEB}(G)$ the set of all UEB subsets of $\mathrm{RUC}_{b}(G)$. For $g \in G$, let $\lambda_{g} \colon G \to G, \, x \mapsto gx$. Recall that $G$ is said to be \emph{amenable} if $\mathrm{RUC}_{b}(G)$ admits a left-invariant mean, i.e., a positive linear map $\mu \colon \mathrm{RUC}_{b}(G) \to \mathbb{R}$ such that $\mu (\mathbf{1}) = 1$ and $\mu (f \circ \lambda_{g}) = \mu (f)$ for all $f \in \mathrm{RUC}_{b}(G)$ and~$g \in G$. By a well-known result of Rickert~\cite[Theorem~4.2]{rickert}, $G$ is amenable if and only if every continuous action of $G$ by affine homeomorphisms on a non-void compact convex subset of a locally convex topological vector space admits a fixed point. We will need the following characterization of amenability in terms of almost invariant finitely supported probability measures from~\cite{SchneiderThom}.

\begin{thm}[\cite{SchneiderThom}, Theorem~3.2]\label{theorem:topological.day} A topological group $G$ is amenable if and only if, for every $\epsilon > 0$, every $H \in \mathrm{RUEB}(G)$ and every finite subset $E \subseteq G$, there exists a finitely supported probability measure $\mu$ on $G$ such that \begin{displaymath}
	\forall g \in E \colon \quad \sup\nolimits_{f \in H} \vert \mathbb{E}_{\mu}(f) - \mathbb{E}_{\mu}(f \circ \lambda_{g}) \vert \leq \epsilon .
	\end{displaymath} \end{thm}

Let us recall that a topological group $G$ is \emph{extremely amenable} if every compact $G$-space has a fixed point. The following definition of a stronger concept is taken from~\cite{pestov10}, but has its origins in~\cite{GlasnerTsirelsonWeiss,GlasnerWeiss}.

\begin{definition} A topological group $G$ is called \emph{whirly amenable} if \begin{enumerate}
	\item[$\bullet$] $G$ is amenable, and
	\item[$\bullet$] every invariant regular Borel probability measure on a compact $G$-space is supported on the set of fixed points.
\end{enumerate} \end{definition}

Clearly, any whirly amenable topological group is extremely amenable. The converse is not true: $\Aut (\mathbb{Q},{<})$ endowed with the topology of point-wise convergence is extremely amenable~\cite{pestov98}, but not whirly amenable~\cite[Remark~1.3]{GlasnerTsirelsonWeiss}.

\section{L\'evy families and concentration of measure}

In this section we recall some results concerning the concentration of measure~\cite{levy,milman,MilmanSchechtman}. The measure concentration phenomenon was linked to extreme amenability by Gromov and Milman~\cite{GromovMilman}. We will largely follow the presentation of~\cite{pestov02,pestovbook,giordanopestov,pestov10} and start off with some generalities on L\'evy families of metric spaces with measures. For more details on measure concentration we refer to~\cite{ledoux}.

\begin{definition}\label{definition:concentration.2} Let $\mathscr{X} = (X,d,\mu)$ be an \emph{$mm$-space}, i.e., $(X,d)$ is a metric space and $\mu$ is a Borel probability measure on $(X,d)$. The \emph{concentration function} of $\mathscr{X}$, $\alpha_{\mathscr{X}} \colon [0,\infty) \to [0,1/2]$, is defined by \begin{displaymath}
	\alpha_{\mathscr{X}}(\epsilon) \defeq \begin{cases}
	1/2 & \text{if } \epsilon = 0 , \\
	1 - \inf \{ \mu (B_{d}(A,\epsilon)) \mid A \subseteq X \text{ Borel, } \, \mu (A) \geq 1/2 \} & \text{otherwise.}
	\end{cases} 
\end{displaymath} A net $(X_{i},d_{i},\mu_{i})_{i \in I}$ of $mm$-spaces is said to \emph{concentrate} or to be a \emph{L\'evy family} if, for every family of Borel subsets $A_{i} \subseteq X_{i}$ ($i \in I$) with $\liminf_{i \in I} \mu_{i}(A_{i}) > 0$, one has $\lim_{i \in I} \mu_{i}(B_{d_{i}}(A_{i},\epsilon)) = 1$ for any $\epsilon > 0$. \end{definition}

\begin{remark}[\cite{GromovMilman}]\label{remark:concentration1} A net $\mathscr{X}_{i} = (X_{i},d_{i},\mu_{i})$ ($i \in I$) of $mm$-spaces is a L\'evy family if and only if $\lim_{i \in I} \alpha_{\mathscr{X}_{i}}(\epsilon ) = 0$ for every $\epsilon > 0$. \end{remark}

We will need the following important example of L\'evy families in the final section.

\begin{exmpl}\label{example:concentration} Let $(X,\mu)$ be a discrete probability space. For an integer $n \geq 1$, we obtain an $mm$-space $\mathscr{Y}_{X,\mu,n} \defeq (X^{n},\mu^{\otimes n},d_{n})$ by equipping $X^{n}$ with the product measure $\mu^{\otimes n}$ and the \emph{normalized Hamming distance} given by \begin{displaymath}
	d_{n}(x,y) \defeq \frac{\vert \{ i \in \{ 1,\ldots,n\} \mid x_{i} \ne y_{i} \} \vert}{n} \qquad (x,y \in X^{n}) .
\end{displaymath} Then \begin{displaymath}
	\alpha_{\mathscr{Y}_{X,\mu,n}}(\epsilon ) \leq 2\exp (-\epsilon^{2}n) 
\end{displaymath} for all $\epsilon \geq 0$ (see~\cite[Proposition~2.1.1]{talagrand} and~\cite{schechtman82,MilmanSchechtman} for finite $X$). In view of Remark~\ref{remark:concentration1} it follows that, if $(X_{i},\mu_{i})_{i \in I}$ is a net of discrete probability spaces and $(n_{i})_{i \in I}$ a corresponding net of positive integers with $n_{i} \to \infty$, then $(\mathscr{Y}_{X_{i},\mu_{i},n_{i}})_{i \in I}$ is a L\'evy family. \end{exmpl}

Now we turn our attention towards measure concentration in uniform spaces and consequences for UEB sets of functions.

\begin{definition}\label{definition:concentration.1} Let $X$ be a uniform space. For an entourage $U$ in $X$, let \begin{displaymath}
	U[A] \defeq \{ y \in X \mid \exists x \in A \colon \, (x,y) \in U \} \qquad (A \subseteq X) .
\end{displaymath} A net $(\mu_{i})_{i \in I}$ of Borel probability measures on $X$ is said to \emph{concentrate in $X$} if, whenever $(A_{i})_{i \in I}$ is a family of Borel subsets of $X$ with $\liminf_{i \in I} \mu_{i}(A_{i}) > 0$, one has \begin{displaymath}
	\lim\nolimits_{i \in I} \mu_{i}(U[A_{i}]) = 1
\end{displaymath} for every open entourage $U$ in $X$. \end{definition}

\begin{remark}[see~\cite{GromovMilman}, 2.1; \cite{pestov02}, Lemma~2.7]\label{remark:concentration2} Let $(X_{i},d_{i},\mu_{i})_{i \in I}$ be a L\'evy family of $mm$-spaces, let $Y$ be a uniform space, and let $f_{i} \colon X_{i} \to Y$ for each $i \in I$. If the family $(f_{i})_{i \in I}$ is uniformly equicontinuous, i.e., for every entourage $U$ of $Y$ there exists $\epsilon > 0$ such that \begin{displaymath}
	\forall i \in I \, \forall x,y \in X_{i} \colon \quad d_{i}(x,y) \leq \epsilon \, \Longrightarrow \, (f_{i}(x),f_{i}(y)) \in U ,
\end{displaymath} then the net $((f_{i})_{\ast}(\mu_{i}))_{i \in I}$ of push-forward measures concentrates in $X$. \end{remark}

Let us recall that every measurable real-valued function $f \colon X \to \mathbb{R}$ on a probability measure space $(X,\mathscr{B},\mu)$ admits a (not necessarily unique) \emph{median}, i.e., a real number $m \in \mathbb{R}$ with \begin{displaymath}
	\mu (\{ x \in X \mid f(x) \geq m \}) \geq \tfrac{1}{2} \leq \mu (\{ x \in X \mid f(x) \leq m \}) .
\end{displaymath} We will need the following well-known fact. 

\begin{lem}[\cite{GromovMilman}, 2.5]\label{lemma:concentration} Let $X$ be a uniform space. Let $(\mu_{i})_{i \in I}$ be net of Borel probability measures on $X$, concentrating in $X$. Let $H \in \mathrm{UEB}(X)$. For each pair $(i,f) \in I \times H$, let $m_{i}(f)$ be a median of $f$ with respect to $\mu_{i}$. Then, for every $\epsilon > 0$, \begin{displaymath}
	\sup\nolimits_{f \in H} \mu_{i}(\{ x \in X \mid \vert f(x) - m_{i}(f) \vert > \epsilon \}) \to 0 .
\end{displaymath} In particular, $\sup\nolimits_{f \in H} \vert \mathbb{E}_{\mu_{i}}(f) - m_{i}(f) \vert \to 0$. \end{lem}

\begin{proof} Let $\epsilon > 0$. Since $H \in \mathrm{UEB}(X)$, there exists a symmetric open entourage $U$ in $X$ such that $\vert f(x) - f(y) \vert \leq \epsilon$ whenever $f \in H$ and $(x,y) \in U$. For every $f \in H$, we conclude that $U[A_{i}(f)] \subseteq B_{i}(f)$ and $U[A'_{i}(f)] \subseteq B'_{i}(f)$ where \begin{align*}
	& A_{i}(f) \defeq \{ x \in X \mid f(x) \leq m_{i}(f) \} , & B_{i}(f) \defeq \{ x \in X \mid f(x) \leq m_{i}(f) + \epsilon \} , \\
	& A'_{i}(f) \defeq \{ x \in X \mid f(x) \geq m_{i}(f) \} , & B'_{i}(f) \defeq \{ x \in X \mid f(x) \geq m_{i}(f) - \epsilon \} .
\end{align*} Hence, $U[A_{i}(f)] \cap U[A'_{i}(f)] \subseteq B_{i}(f) \cap B'_{i}(f)$ for all $f \in H$. We need to show that \begin{displaymath}
	\sup\nolimits_{f \in H} \mu_{i}(X\setminus (B_{i}(f) \cap B'_{i}(f))) \to 0 .
\end{displaymath} Let $\delta > 0$. For each pair $(i,f) \in I \times H$, our hypothesis on $m_{i}(f)$ asserts that \begin{displaymath}
	\min \{ \mu_{i}(A_{i}(f)),\, \mu_{i}(A'_{i}(f)) \} \geq \tfrac{1}{2} .
\end{displaymath} Since $(\mu_{i})_{i \in I}$ concentrates in $X$, we thus find some $i_{0} \in I$ such that \begin{displaymath}
	\forall i \in I , \, i \geq i_{0} \, \forall f \in H \colon \quad \min \{ \mu_{i}(U[A_{i}(f)]), \, \mu_{i}(U[A_{i}'(f)]) \} \geq 1 - \tfrac{\delta}{2} .
\end{displaymath} Consequently, if $i \in I$ with $i \geq i_{0}$, then \begin{displaymath}
	\mu_{i} (B_{i}(f) \cap B'_{i}(f)) \geq \mu_{i}(U[A_{i}(f)] \cap U[A_{i}'(f)]) \geq 1 - \delta
\end{displaymath} for every $f \in H$, which means that $\sup\nolimits_{f \in H} \mu_{i}(X\setminus (B_{i}(f) \cap B'_{i}(f))) \leq \delta$. This proves the first assertion of the lemma.

To deduce the second statement, consider any $\epsilon > 0$. For $(i,f) \in I \times H$, let $D_{i}(f) \defeq \left\{ x \in X \left| \, \vert f(x) - m_{i}(f) \vert > \tfrac{\epsilon}{2} \right\} \right.$. By the above, there is $i_{0} \in I$ such that \begin{displaymath}
	\forall i \in I , \, i \geq i_{0} \colon \quad \mu_{i}(D_{i}(f)) \leq \frac{\epsilon}{4\left(\sup\nolimits_{f \in H} \Vert f \Vert_{\infty} + 1\right)} .
\end{displaymath} Now, if $i \in I$ with $i \geq i_{0}$, then \begin{align*}
	\vert \mathbb{E}_{\mu_{i}}(f)& - m_{i}(f) \vert \, \leq \, \int_{X} \vert f(x) - m_{i}(f) \vert \, d\mu_{i}(x) \\
	&\leq \, \int_{X \setminus D_{i}(f)} \vert f(x) - m_{i}(f) \vert \, d\mu_{i}(x) \, + \, \int_{D_{i}(f)} \vert f(x) - m_{i}(f) \vert \, d\mu_{i}(x) \\
	&\leq \, \tfrac{\epsilon}{2} + 2\Vert f \Vert_{\infty} \mu_{i}(D_{i}(f)) \, \leq \, \tfrac{\epsilon}{2} + \tfrac{\epsilon}{2} \, = \, \epsilon
\end{align*} for every $f \in H$, i.e., $\sup_{f \in H} \vert \mathbb{E}_{\mu_{i}}(f) - m_{i}(f) \vert \leq \epsilon$. This shows that \begin{displaymath}
	\sup\nolimits_{f \in H} \vert \mathbb{E}_{\mu_{i}}(f) - m_{i}(f) \vert \to 0 .\qedhere
\end{displaymath} \end{proof}

\begin{cor}\label{corollary:concentration} Let $X$ be a uniform space and let $(\mu_{i})_{i \in I}$ be a net of Borel probability measures on $X$, concentrating in $X$. Let $H \in \mathrm{UEB}(X)$ and $\epsilon > 0$. Then \begin{displaymath}
	\sup\nolimits_{f \in H} \mu_{i} (\{ x \in X \mid \vert f(x) - \mathbb{E}_{\mu_{i}}(f) \vert > \epsilon \}) \to 0 .
\end{displaymath} \end{cor}

\begin{proof} For each pair $(i,f) \in I \times H$, let $m_{i}(f)$ be a median of $f$ with respect to $\mu_{i}$. Let $\epsilon,\delta >0$. We need to show that $\sup\nolimits_{f \in H} \mu_{i} (\{ x \in X \mid \vert f(x) - \mathbb{E}_{\mu_{i}}(f) \vert > \epsilon \}) \leq \delta$. By Lemma~\ref{lemma:concentration}, there exists $i_{0} \in I$ such that, for every $i \in I$ with $i \geq i_{0}$, both \begin{displaymath}
	\sup\nolimits_{f \in H} \mu_{i}\left(\left\{ x \in X \left\vert \vert f(x) - m_{i}(f) \vert > \tfrac{\epsilon}{2} \right\}\right) \leq \delta \right.
\end{displaymath} and $\sup_{f \in H} \vert \mathbb{E}_{\mu_{i}}(f) - m_{i}(f) \vert \leq \tfrac{\epsilon}{2}$. Hence, if $i \in I$ with $i \geq i_{0}$, then \begin{displaymath}
	\mu_{i}(\{ x \in X \mid \vert f(x) - \mathbb{E}_{\mu_{i}}(f) \vert > \epsilon \}) \leq \mu_{i}\left(\left\{ x \in X \left| \, \vert f(x) - m_{i}(f) \vert > \tfrac{\epsilon}{2} \right\}\right) \leq \delta \right.
\end{displaymath} for all $f \in H$, as desired. \end{proof}

Utilizing the corollary above, we obtain a sufficient criterion for whirly amenability of topological groups in terms of measure concentration. Let us agree on some additional terminology.

\begin{definition} Let $G$ be a topological group. A net $(\mu_{i})_{i \in I}$ of Borel probability measures on $G$ is said to \emph{converge to invariance in the UEB topology (over $G$)} if \begin{displaymath}
	\sup\nolimits_{f \in H} \vert \mathbb{E}_{\mu_{i}}(f) - \mathbb{E}_{\mu_{i}}(f \circ \lambda_{g}) \vert \to 0
\end{displaymath} for all $H \in \mathrm{RUEB}(G)$ and $g \in G$. \end{definition}

The following result for Polish groups was proven in~\cite[Theorem~5.7]{pestov10}. We include a full proof of this slight generalization for the sake of convenience.

\begin{thm}[\cite{pestov10}, Theorem~5.7]\label{theorem:whirly.groups} Let $G$ be a topological group. If there exists a net $(\mu_{i})_{i \in I}$ of Borel probability measures on $G$ such that \begin{enumerate}
	\item[$(1)$] $(\mu_{i})_{i \in I}$ concentrates in $G$,
	\item[$(2)$] $(\mu_{i})_{i \in I}$ converges to invariance in the UEB topology,
\end{enumerate} then $G$ is whirly amenable. \end{thm}

\begin{proof} Let $(\mu_{i})_{i \in I}$ be a net as above. Then Theorem~\ref{theorem:topological.day} asserts that $G$ is amenable: note that the implication ($\Longleftarrow$) in Theorem~\ref{theorem:topological.day} is valid even without the measures $(\mu_{i})_{i \in I}$ being finitely supported. To see that $G$ is actually whirly amenable, we will follow the lines of~\cite[Proof of Theorem~5.7]{pestov10}. Consider a compact $G$-space $X$ and let $\mu$ by a $G$-invariant regular Borel probability measure on $X$. We want to show that $\mu$ is supported on the set of $G$-fixed points. By a standard application of Urysohn's lemma, it suffices to prove that \begin{displaymath}
	\forall g \in G \, \forall f \in \mathrm{C}(X,[0,1]) \colon \quad \int \vert f(x) - f(gx) \vert \, d\mu (x) = 0 .
\end{displaymath} Let $f \in \mathrm{C}(X,[0,1])$, $g \in G$ and $\epsilon > 0$. For $x \in X$, let $f_{x} \colon G \to [0,1], \, h \mapsto f(hx)$. We observe that $\{ f_{x} \mid x \in X \} \in \mathrm{RUEB}(G)$. To see this, let $\delta > 0$. Since the map $G \times X \to \mathbb{R}, \, (g,x) \mapsto f(gx)$ is continuous, for each $x \in X$ we find $U_{x} \in \mathscr{U}(G)$ and an open neighborhood $V_{x}$ of $x$ in $X$ such that $\diam f(U_{x}V_{x}) \leq \delta$. By compactness of $X$, there exist $U_{1},\ldots,U_{n} \in \mathscr{U}(G)$ and open sets $V_{1},\ldots,V_{n} \subseteq X$ such that $\bigcup_{i=1}^{n} V_{i} = X$ and $\diam f(U_{i}V_{i}) \leq \delta$ for each $i \in \{ 1,\ldots,n \}$. Of course, $U \defeq \bigcap_{i=1}^{n}U_{i} \in \mathscr{U}(G)$. Moreover, if $g,h \in G$ and $gh^{-1} \in U$, then \begin{displaymath}
	\vert f_{x}(g) - f_{x}(h) \vert = \vert f((gh^{-1})(hx)) - f(hx) \vert \leq \delta
\end{displaymath} for all $x \in X$, which shows that $\{ f_{x} \mid x \in X \}$ indeed belongs to $\mathrm{RUEB}(G)$. Hence, by~(2), there exists some $i_{0} \in I$ such that, for all $i \in I$ with $i \geq i_{0}$, \begin{displaymath}
	\sup\nolimits_{x \in X} \vert \mathbb{E}_{\mu_{i}} (f_{x}) - \mathbb{E}_{\mu_{i}}(f_{x} \circ \lambda_{g}) \vert \leq \tfrac{\epsilon}{5} .
\end{displaymath} Furthermore, since both $\{ f_{x} \mid x \in X \}$ and $\{ f_{x} \circ \lambda_{g} \mid x \in X \}$ belong to $\mathrm{RUEB}(G)$, we may apply~(1) and Corollary~\ref{corollary:concentration} find some $i_{1} \in I$ with $i_{1} \geq i_{0}$ such that, for all $i \in I$ with $i \geq i_{1}$, we have \begin{displaymath}
	\sup\nolimits_{x \in X} \mu_{i} \left(\left\{ h \in G \left| \, \vert f_{x}(h) - \mathbb{E}_{\mu_{i}}(f_{x}) \vert > \tfrac{\epsilon}{5} \right\}\right) \leq \tfrac{\epsilon}{5} \right.
\end{displaymath} as well as \begin{displaymath}
	\sup\nolimits_{x \in X} \mu_{i} \left(\left\{ h \in G \left| \, \vert f_{x}(gh) - \mathbb{E}_{\mu_{i}}(f_{x} \circ \lambda_{g}) \vert > \tfrac{\epsilon}{5} \right\}\right) \leq \tfrac{\epsilon}{5} . \right.
\end{displaymath} From the assertions above, we deduce that \begin{align*}
	\int \vert & f(hx) - f(ghx) \vert \, d\mu_{i} (h) \\
	&\leq \, \int \vert f_{x}(h) - \mathbb{E}_{\mu_{i}}(f_{x}) \vert \, d\mu_{i}(h) + \tfrac{\epsilon}{5} + \int \vert \mathbb{E}_{\mu_{i}}(f_{x} \circ \lambda_{g}) - f_{x}(gh) \vert \, d\mu_{i}(h) \\
	&\leq \, \tfrac{2\epsilon}{5} + \tfrac{\epsilon}{5} + \tfrac{2\epsilon}{5} \, = \, \epsilon 
\end{align*} for all $x \in X$ and $i \in I$, $i \geq i_{1}$. Thus, by $G$-invariance of $\mu$ and Fubini's theorem, \begin{align*}
	\int \vert f(x) - f(gx) \vert \, d\mu (x) \, &= \int \int \vert f(hx) - f(ghx) \vert \, d\mu (x) \, d\mu_{i_{1}}(h) \\
	&= \int \int \vert f(hx) - f(ghx) \vert \, d\mu_{i_{1}} (h) \, d\mu(x) \, \leq \, \epsilon .
\end{align*} This completes the argument. \end{proof}

\section{Proving Theorem~\ref{theorem:main}}\label{section:final}

In this section we shall prove Theorem~\ref{theorem:main}. Let $G$ be a topological group. As stated in the introduction, we endow $L^{0}(G)$ with the topology of convergence in measure. More explicitly, the sets of the form \begin{displaymath}
	N(U,\epsilon) \defeq \{ f \in L^{0}(G) \mid \lambda (\{ x \in [0,1] \mid f(x) \notin U \}) < \epsilon \} \quad (U \in \mathscr{U}_{o}(G), \, \epsilon > 0)
\end{displaymath} constitute a neighborhood basis of the neutral element in $L^{0}(G)$. Given a natural number $n \geq 1$, let us define $h_{n} \colon G^{n} \to L^{0}(G)$ by setting \begin{displaymath}
	h_{n}(g)\vert_{[(i-1)/n,i/n)} \equiv g_{i} \qquad (i \in \{ 1,\ldots,n \}) 
\end{displaymath} for every $g \in G^{n}$. Of course, $h_{n}$ is an embedding of the topological group $G^{n}$ into $L^{0}(G)$ for every $n \geq 1$. Moreover, we note the following fact.

\begin{lem}\label{lemma:density} Let $G$ be a topological group. Let $f \in L^{0}(G)$ and $U \in \mathscr{U}(L^{0}(G))$. Then there exist a finite subset $E \subseteq G$ and $m \in \mathbb{N}$ such that \begin{displaymath}
	\forall n \in \mathbb{N}, \, n \geq m \colon \quad f \in Uh_{n}(E^{n}) .
\end{displaymath} \end{lem}

\begin{proof} Let $V \in \mathscr{U}_{o}(G)$ and $\epsilon > 0$ with $N(V,\epsilon) \subseteq U$. As $f$ is strongly $\lambda$-measurable, there is a closed subset $C \subseteq [0,1]$ such that $\lambda ([0,1]\setminus C) < \epsilon$ and $f\vert C$ is continuous. Since $C$ is compact, $f(C)$ is a compact subset of $G$, and $f\vert_{C} \colon C \to G$ is uniformly continuous. Choose $W \in \mathscr{U}(G)$ such that $W^{2} \subseteq V$. Due to $f(C)$ being compact, there exists a finite subset $E \subseteq G$ with $f(C) \subseteq WE$ and $e \in E$. By uniform continuity of $f\vert_{C}$, there furthermore exists a natural number $m \geq 1$ such that \begin{displaymath}
	\forall x,y \in C \colon \quad \vert x-y \vert \leq \tfrac{1}{m} \, \Longrightarrow \, f(x)f(y)^{-1} \in W .
\end{displaymath} We claim that the pair $(E,m)$ is as desired. To see this, let $n \in \mathbb{N}$, $n \geq m$. Put \begin{displaymath}
	I \defeq \left\{ i \in \{ 1,\ldots,n \} \left\vert \, \left[\tfrac{i-1}{n},\tfrac{i}{n}\right) \cap C \ne \emptyset \right\} . \right.
\end{displaymath} For each $i \in I$, pick some element $x_{i} \in \left[\tfrac{i-1}{n},\tfrac{i}{n}\right) \cap C$ and then choose any $g_{i} \in E$ with $f(x_{i}) \in Wg_{i}$. Extend $(g_{i})_{i \in I}$ to a tuple $g \in E^{n}$ by setting $g_{i} \defeq e$ for each $i \in \{ 1,\ldots,n \} \setminus I$. We claim that $f \in Uh_{n}(g)$. For each $x \in C\setminus \{ 1 \}$, there exists some $i \in \{ 1,\ldots,n \}$ with $x \in \left[ \tfrac{i-1}{n},\tfrac{i}{n} \right)$, which implies that $i \in I$ and thus \begin{displaymath}
	f(x) \in Wf(x_{i}) \subseteq WWg_{i} \subseteq Vg_{i} = V h_{n}(g)(x) .
\end{displaymath} This shows that \begin{displaymath}
	B \defeq \left\{ x \in [0,1] \left| f(x)h_{n}(g)(x)^{-1} \notin V \right\} \right. \subseteq \, [0,1]\setminus C . 
\end{displaymath} Therefore, $\lambda (B) \leq \lambda ([0,1]\setminus C) < \epsilon$ and so $fh_{n}(g)^{-1} \in N(V,\epsilon) \subseteq Uh_{n}(g)$. \end{proof}

\begin{cor}\label{corollary:density} Let $G$ be a topological group. Then $L^{0}(G) = \overline{\bigcup \{ h_{n}(G^{n}) \mid n\geq 1 \}}$. \end{cor}

Furthermore, we will need the subsequent simple observation about UEB sets. Let $G$ be a topological group. For $n \in \mathbb{N}\setminus \{ 0 \}$, $i \in \{ 1,\ldots,n \}$ and $a \in G^{n-1}$, let \begin{displaymath}
	c_{i,a} \colon G \to G^{n} , \quad x \mapsto (a_{1},\ldots ,a_{i-1},x,a_{i},\ldots,a_{n-1}) .
\end{displaymath} Moreover, for a subset $H \subseteq \mathrm{RUC}_{b}(L^{0}(G))$, let \begin{displaymath}
	[H] \defeq \left\{ f \circ h_{n} \circ c_{i,a} \left\vert \, f \in H, \, n \in \mathbb{N}\setminus \{ 0 \}, \, i \in \{ 1,\ldots,n \}, \, a \in G^{n-1} \right\} . \right.
\end{displaymath} The following is an immediate consequence of the definitions.

\begin{lem}\label{lemma:ueb} Let $G$ be a topological group. Then \begin{displaymath}
	\forall H \in \mathrm{RUEB}(L^{0}(G)) \colon \quad [H] \in \mathrm{RUEB}(G) .
\end{displaymath} \end{lem}

\begin{proof} Let $H \in \mathrm{RUEB}(L^{0}(G))$. Evidently, $[H]$ is norm-bounded as $H$ is. To prove uniform equicontinuity, let $\epsilon > 0$. Since $H \in \mathrm{RUEB}(L^{0}(G))$, there is $U \in \mathscr{U}(L^{0}(G))$ such that $\vert f(x) - f(y) \vert \leq \epsilon$ for all $f \in H$ and $x,y \in L^{0}(G)$ with $xy^{-1} \in U$. Fix any $V \in \mathscr{U}_{o}(G)$ and $\delta > 0$ so that $N(V,\delta) \subseteq U$. We claim that $\vert f'(x) - f'(y) \vert \leq \epsilon$ for all $f' \in [H]$ and $x,y \in G$ with $xy^{-1} \in V$. Let $n \in \mathbb{N}\setminus \{ 0 \}$, $i \in \{ 1,\ldots,n \}$, $a \in G^{n-1}$, and $f \in H$. For any $x,y \in G$ with $xy^{-1} \in V$, it follows that \begin{displaymath}
	h_{n}(c_{i,a}(x))h_{n}(c_{i,a}(y))^{-1} = h_{n}\left( c_{i,a}(x)c_{i,a}(y)^{-1} \right) = h_{n}\left( c_{i,(e,\ldots,e)}\left(xy^{-1}\right) \right) \in N(V,\delta)
\end{displaymath} and therefore $\vert f(h_{n}(c_{i,a}(x))) - f(h_{n}(c_{i,a}(y))) \vert \leq \epsilon$. This proves our claim and hence shows that $[H]$ belongs to $\mathrm{RUEB}(G)$. \end{proof}

Next we give the central argument for the proof of Theorem~\ref{theorem:main}.

\begin{lem}\label{lemma:convergence.to.invariance} Let $G$ be a topological group. Suppose that $(n_{i},\mu_{i})_{i \in I}$ is a net in $(\mathbb{N}\setminus \{ 0 \}) \times \mathrm{Pr}_{\mathrm{B}}(G)$ such that $n_{i} \to \infty$ and \begin{displaymath}
	\forall H \in \mathrm{RUEB}(G) \, \forall g \in G \colon \quad  \sup\nolimits_{f \in H} \vert \mathbb{E}_{\mu_{i}}(f) - \mathbb{E}_{\mu_{i}}(f \circ \lambda_{g}) \vert \cdot n_{i} \to 0 .
\end{displaymath} Then $\left( ( h_{n_{i}})_{\ast}\left(\mu_{i}^{\otimes n_{i}}\right) \right)_{i \in I}$ converges to invariance in the UEB topology (over~$L^{0}(G)$). \end{lem}

\begin{proof} For each $i \in I$, consider the push-forward $\nu_{i} \defeq ( h_{n_{i}})_{\ast}\left(\mu_{i}^{\otimes n_{i}}\right) \in \mathrm{Pr}_{\mathrm{B}}(L^{0}(G))$. Let $H \in \mathrm{RUEB}(L^{0}(G))$, $g \in L^{0}(G)$ and $\epsilon > 0$. Since $H \in \mathrm{RUEB}(L^{0}(G))$, there exists $U \in \mathscr{U}(L^{0}(G))$ such that $\Vert f - (f \circ \lambda_{u}) \Vert_{\infty} \leq \tfrac{\epsilon}{2}$ for all $f \in H$ and $u \in U$. By Lemma~\ref{lemma:density}, there exist a finite subset $E \subseteq G$ and $m \in \mathbb{N}$ such that \begin{displaymath}
	\forall n \in \mathbb{N}, \, n \geq m \colon \quad g \in Uh_{n}(E^{n}) .
\end{displaymath} Without loss of generality, suppose that $e \in E$. By Lemma~\ref{lemma:ueb} and our assumptions, there exists $i_{0} \in I$ such that, for all $i \in I$ with $i \geq i_{0}$, we have $n_{i} \geq m$ and \begin{displaymath}\tag{$\ast$}\label{assumption}
	\forall s \in E \colon \quad \sup\nolimits_{f \in [H]} \vert \mathbb{E}_{\mu_{i}}(f) - \mathbb{E}_{\mu_{i}}(f \circ \lambda_{s}) \vert \leq \tfrac{\epsilon}{2n_{i}} .
\end{displaymath} We claim that \begin{equation}\tag{$\ast \ast$}\label{claim}
	\forall i \in I , \, i \geq i_{0} \colon \quad \sup\nolimits_{f \in H} \vert \mathbb{E}_{\nu_{i}}(f) - \mathbb{E}_{\nu_{i}}(f \circ \lambda_{g}) \vert \leq \epsilon .
\end{equation} So, let $i \in I$ with $i \geq i_{0}$. Since $n_{i} \geq m$, we find $g' \in E^{n_{i}}$ such that $g = Uh_{n_{i}}(g')$. For each $j \in \{ 1,\ldots,n_{i} \}$, let \begin{align*}
	& a_{j} \defeq \left(g'_{1},\ldots,g'_{j},e,\ldots,e\right) \in E^{n_{i}} , & b_{j} \defeq \left(g'_{1},\ldots,g'_{j-1},e,\ldots,e\right) \in E^{n_{i}-1} .
\end{align*} Also, let $a_{0} \defeq (e,\ldots,e) \in E^{n_{i}}$. For any $j \in \{ 1,\ldots,n_{i}\}$ and $z \in G^{n_{i}-1}$, we observe that $\lambda_{a_{j}} \circ c_{j,z} = c_{j,b_{j}z} \circ \lambda_{g'_{j}}$ and $\lambda_{a_{j-1}} \circ c_{j,z} = c_{j,b_{j}z}$. Combined with~\eqref{assumption} and Fubini's theorem, this implies that \begin{align*}
	\bigl\vert \mathbb{E}_{\nu_{i}}\bigl(f &\circ \lambda_{h_{n_{i}}(a_{j-1})}\bigr) - \mathbb{E}_{\nu_{i}}\bigl(f \circ \lambda_{h_{n_{i}}(a_{j})}\bigr) \bigr\vert \\
	&= \, \left\vert \int \bigl(f \circ \lambda_{h_{n_{i}}(a_{j-1})} \bigr) - \bigl(f \circ \lambda_{h_{n_{i}}(a_{j})} \bigr) \, d\nu_{i} \right\vert \\
	&= \, \left\vert \int \left(f \circ h_{n_{i}} \circ \lambda_{a_{j-1}}\right) - \left(f \circ h_{n_{i}} \circ \lambda_{a_{j}}\right) \, d\mu_{i}^{\otimes n_{i}} \right\vert \\
	&= \, \left\vert \int \mathbb{E}_{\mu_{i}} \!\left(f \circ h_{n_{i}} \circ \lambda_{a_{j-1}} \circ c_{j,z}\right) - \mathbb{E}_{\mu_{i}} \!\left(f \circ h_{n_{i}} \circ \lambda_{a_{j}}\circ c_{j,z}\right) \, d\mu_{i}^{\otimes (n_{i} -1)}(z) \right\vert \\
	&= \, \left\vert \int \mathbb{E}_{\mu_{i}} \!\left(f \circ h_{n_{i}} \circ c_{j,b_{j}z}\right) - \mathbb{E}_{\mu_{i}} \!\left(f \circ h_{n_{i}} \circ c_{j,b_{j}z} \circ \lambda_{g'_{j}} \right) \, d\mu_{i}^{\otimes (n_{i} -1)}(z) \right\vert \\
	&\leq \, \int \left\vert \mathbb{E}_{\mu_{i}}\!\left(f \circ h_{n_{i}} \circ c_{j,b_{j}z}\right) - \mathbb{E}_{\mu_{i}}\!\left(f \circ h_{n_{i}} \circ c_{j,b_{j}z} \circ \lambda_{g'_{j}} \right) \right\vert \, d\mu_{i}^{\otimes (n_{i} -1)}(z) \\
	&\leq \int \tfrac{\epsilon}{2n_{i}} \, d\mu_{i}^{\otimes (n_{i}-1)}(z) \, = \,  \tfrac{\epsilon}{2n_{i}}
\end{align*} for all $j \in \{ 1,\ldots,n_{i} \}$ and $f \in H$. We conclude that, for all $f \in H$, \begin{align*}
	\bigl\vert &\mathbb{E}_{\nu_{i}}(f) - \mathbb{E}_{\nu_{i}}\bigl(f \circ \lambda_{h_{n_{i}}(g')}\bigr) \bigr\vert \\
	&\leq \, \bigl\vert \mathbb{E}_{\nu_{i}}(f) - \mathbb{E}_{\nu_{i}}\bigl(f \circ \lambda_{h_{n_{i}}(a_{1})}\bigr) \bigr\vert + \ldots + \bigl\vert \mathbb{E}_{\nu_{i}}\bigl(f \circ \lambda_{h_{n_{i}}(a_{n_{i}-1})}\bigr) - \mathbb{E}_{\nu_{i}}\bigl(f \circ \lambda_{h_{n_{i}}(g')}\bigr) \bigr\vert \\
	&\leq \, \tfrac{\epsilon}{2n_{i}} + \ldots + \tfrac{\epsilon}{2n_{i}} \, = \, \tfrac{\epsilon}{2} ,
\end{align*} and therefore \begin{align*}
	\vert \mathbb{E}_{\nu_{i}}(f) \, - \,  &\mathbb{E}_{\nu_{i}}(f \circ \lambda_{g}) \vert \\
		&\leq \, \bigl\vert \mathbb{E}_{\nu_{i}}(f) - \mathbb{E}_{\nu_{i}}\bigl(f \circ \lambda_{h_{n_{i}}(g')}\bigr) \bigr\vert + \bigl\vert \mathbb{E}_{\nu_{i}}\bigl(f \circ \lambda_{h_{n_{i}}(g')}\bigr) - \mathbb{E}_{\nu_{i}}(f \circ \lambda_{g}) \bigr\vert \\
		&\leq \, \tfrac{\epsilon}{2} + \bigl\Vert \bigl(f \circ \lambda_{h_{n_{i}}(g')}\bigr) - (f \circ \lambda_{g}) \bigr\Vert_{\infty} \\
		&= \, \tfrac{\epsilon}{2} + \bigl\Vert f - \bigl(f \circ \lambda_{gh_{n_{i}}(g')^{-1}}\bigr) \bigr\Vert_{\infty} \, \leq \, \epsilon ,
\end{align*} which proves~\eqref{claim} and hence completes the argument. \end{proof}

Now everything is prepared to prove the main result.

\begin{proof}[Proof of Theorem~\ref{theorem:main}] The implications (4)$\Longrightarrow$(3)$\Longrightarrow$(2) are trivial.
	
(1)$\Longrightarrow$(4). According to Theorem~\ref{theorem:topological.day}, for all $n \in \mathbb{N}$, $n \geq 1$, $\epsilon > 0$, $H \in \mathrm{RUEB}(G)$ and finite $E \subseteq G$, there exists $\mu \in \mathrm{Pr}_{\mathrm{fin}}(G)$ such that \begin{displaymath}
	\forall g \in E \colon \quad \sup\nolimits_{f \in H} \vert \mathbb{E}_{\mu}(f) - \mathbb{E}_{\mu}(f \circ \lambda_{g}) \vert \leq \tfrac{\epsilon}{n} .
\end{displaymath} Thus, we can choose a net $(n_{i},\mu_{i})_{i \in I}$ in $(\mathbb{N}\setminus \{ 0 \}) \times \mathrm{Pr}_{\mathrm{fin}}(G)$ so that $n_{i} \to \infty$ and \begin{displaymath}
	\forall H \in \mathrm{RUEB}(G) \, \forall g \in G \colon \quad  \sup\nolimits_{f \in H} \vert \mathbb{E}_{\mu_{i}}(f) - \mathbb{E}_{\mu_{i}}(f \circ \lambda_{g}) \vert \cdot n_{i} \to 0 .
\end{displaymath} For each $i \in I$, let us consider $\nu_{i} \defeq ( h_{n_{i}})_{\ast}\left(\mu_{i}^{\otimes n_{i}}\right) \in \mathrm{Pr}_{\mathrm{fin}}(L^{0}(G))$. By Lemma~\ref{lemma:convergence.to.invariance}, the net $(\nu_{i})_{i \in I}$ converges to invariance in the UEB topology.

We will apply Theorem~\ref{theorem:whirly.groups} to the net $(\nu_{i})_{i \in I}$. So, we need to argue that $(\nu_{i})_{i \in I}$ concentrates in $L^{0}(G)$. For each $i \in I$, denote by $S_{i}$ the (finite) support of $\mu_{i}$ and consider the discrete probability space $(S_{i},\mu_{i})$. By Example~\ref{example:concentration}, the net \begin{displaymath}
	(\mathscr{Y}_{S_{i},\nu_{i},n_{i}})_{i \in I} = \left(S_{i}^{n_{i}},\mu_{i}^{\otimes n_{i}},d_{n_{i}}\right)_{i \in I}
\end{displaymath} concentrates. In view of Remark~\ref{remark:concentration2}, it therefore suffices to show that the family $(h_{n_{i}})_{i \in I}$ is uniformly equicontinuous. For this purpose, let $U \in \mathscr{U}_{o}(G)$ and $\epsilon > 0$. For all $i \in I$ and $x,y \in S_{i}^{n_{i}}$, we have \begin{align*}
	\lambda \left(\left\{ t \in [0,1] \left\vert h_{n_{i}}(x)(t) h_{n_{i}}(y)(t)^{-1} \notin U \right\}\right) \right. &\leq \, \lambda (\{ t \in [0,1] \mid h_{n_{i}}(x)(t) \ne h_{n_{i}}(y)(t) \}) \\
		&= \, \frac{\vert \{ j \in \{ 1,\ldots,n_{i} \} \mid x_{j} \ne y_{j} \} \vert}{n_{i}} \\
		&= \, d_{n_{i}}(x,y) ,
\end{align*} and thus \begin{displaymath}
	d_{n_{i}}(x,y) < \epsilon \quad \Longrightarrow \quad h_{n_{i}}(x)h_{n_{i}}(x)^{-1} \in N(U,\epsilon) .
\end{displaymath} Hence, $(\nu_{i})_{i \in I}$ concentrates in $L^{0}(G)$ by Remark~\ref{remark:concentration2}, and so $L^{0}(G)$ is whirly amenable by Theorem~\ref{theorem:whirly.groups}.

(2)$\Longrightarrow$(1). Consider the operator $\Phi \colon \mathrm{RUC}_{b}(G) \to \mathrm{RUC}_{b}(L^{0}(G))$ given by \begin{displaymath}
	\Phi (f)(h) \defeq \int_{[0,1]} f(h(x)) \, d\lambda (x) \qquad (f \in \mathrm{RUC}_{b}(G), \, h \in L^{0}(G)) .
\end{displaymath} We first check that $\Phi$ is well defined. Let $f \in \mathrm{RUC}_{b}(G)$. For every $h \in L^{0}(G)$, the function $f \circ h \colon [0,1] \to \mathbb{R}$ is bounded and $\mu$-measurable, whence the expression $\Phi (f)(h)$ is well defined. Moreover, \begin{displaymath}
	\sup_{h \in L^{0}(G)} \vert \Phi (f)(h) \vert = \Vert f \Vert_{\infty}
\end{displaymath} and so $\Phi (f) \in \ell^{\infty}(L^{0}(G))$. To see that $\Phi (f)$ belongs to $\mathrm{RUC}_{b}(L^{0}(G))$, let $\epsilon > 0$. Since $f \in \mathrm{RUC}_{b}(G)$, there exists $U \in \mathscr{U}_{o}(G)$ such that $\Vert f - (f \circ \lambda_{g}) \Vert_{\infty} \leq \epsilon/2$ for all $g \in U$. Consider \begin{displaymath}
	\epsilon' \defeq \frac{\epsilon}{2\Vert f \Vert_{\infty} + 1} .
\end{displaymath} Then $V \defeq N(U,\epsilon')$ is an identity neighborhood in $L^{0}(G)$. Now, if $h_{0},h_{1} \in L^{0}(G)$ and $h_{0}h_{1}^{-1} \in V$, then $\mu (A) \leq \epsilon'$ for $A \defeq \{ x \in [0,1] \mid h_{0}(x)h_{1}(x)^{-1} \notin U \}$, and so \begin{align*}
	\vert \Phi (f)&(h_{0}) - \Phi(f)(h_{1}) \vert \\
	&\leq \, \int_{A} \vert f(h_{0}(x)) - f(h_{1}(x)) \vert \, d\lambda (x) + \int_{[0,1]\setminus A} \vert f(h_{0}(x)) - f(h_{1}(x)) \vert \, d\lambda (x) \\
	&\leq \, 2\Vert f \Vert_{\infty} \epsilon' + \frac{\epsilon}{2} \, \leq \, \epsilon .
\end{align*} This shows that $\Phi (f) \in \mathrm{RUC}_{b}(L^{0}(G))$. So, $\Phi$ is a well defined mapping. It is straightforward to check that $\Phi$ is linear, $\Phi (\mathbf{1}_{G}) = \mathbf{1}_{L^{0}(G)}$, and $\Phi (f_{0}) \leq \Phi (f_{1})$ for all $f_{0},f_{1} \in \mathrm{RUC}_{b}(G)$ with $f_{0} \leq f_{1}$. Also, if $f \in \mathrm{RUC}_{b}(G)$ and $g \in G$, then \begin{align*}
	\Phi (f \circ \lambda_{g}) (h) &= \int_{[0,1]} f(gh(x)) \, d\lambda (x) = \int_{[0,1]} f(h_{1}(g)(x)h(x)) \, d\lambda (x) \\
	&= \Phi (f) (h_{1}(g)h) = (\Phi (f) \circ \lambda_{h_{1}(g)})(h)
\end{align*} for all $h \in L^{0}(G)$, that is, $\Phi (f \circ \lambda_{g}) = \Phi (f) \circ \lambda_{h_{1}(g)}$. Now, assuming that $L^{0}(G)$ is amenable and considering a left-invariant mean $\mu$ on $\mathrm{RUC}_{b}(L^{0}(G))$, we deduce from the properties above that $\mu \circ \Phi \colon \mathrm{RUC}_{b}(G) \to \mathbb{R}$ is a left-invariant mean, whence $G$ is amenable. This completes the proof. \end{proof}

We conclude this note with a question about the concentration of invariant means.

\begin{question} Let $(G_n)_{n \in \mathbb{N}}$ be a sequence of amenable topological groups. Then there are invariant means, $\phi_n$, on each product $\prod_{i=1}^n G_i$, that is, invariant regular Borel probability measures on the Samuel compactifications (greatest ambits) $\mathrm{S}\left(\prod_{i=1}^n G_i\right)$. These measures are not, in general, the product measures. Assume for simplicity that $(G_i)_{i \in \mathbb{N}}$ are Polish, and equip the products $\prod_{i=1}^n G_i$ with the (right-invariant) normalized Hamming ($\ell^1$-type sum) distance. Do the means $(\phi_n)_{n \in \mathbb{N}}$ concentrate, in the sense that for every sequence of bounded 1-Lipschitz functions $f_n$ on $\prod_{i=1}^n G_i$, one has \begin{displaymath}
	\phi_{n}\left((f_{n}-\phi_{n}(f_{n}))^2\right)\to 0 \ ?
\end{displaymath} It looks likely that this can be proved using the results of~\cite{SchneiderThom} in combination with the martingale technique as in~\cite{MilmanSchechtman}. \end{question}

\section*{Acknowledgments}

The first-named author was partially supported by the CNPq (Brazil) through a Visiting Researcher Fellowship (processo 310012-2016) and by the Canadian NSERC through a 2012--2017 Disovery grant.
The second-named author acknowledges funding of the German Research Foundation (reference no.~SCHN 1431/3-1) as well as the Excellence Initiative by the German Federal and State Governments.

%%%%%%%%%%%%%%%%%%%%%%%%%%%%%%%%%%

\end{document}